\documentclass[11pt]{amsart}
\usepackage{amssymb,graphicx}
\usepackage{verbatim}
\usepackage{float}
\usepackage[dvipsnames]{xcolor}
\usepackage{amsmath}
\usepackage{cleveref}
\usepackage{autonum}
\usepackage{enumitem}
\usepackage{ulem}
\newtheorem{theorem}{Theorem}[section]
\newtheorem{lemma}[theorem]{Lemma}

\theoremstyle{definition}
\newtheorem{definition}[theorem]{Definition}
\newtheorem*{nota}{Notation}

\newcommand{\R}{{\mathbb R}}

\usepackage{refcount}

\newcounter{cte}

\newif\ifrevisioncolors
\revisioncolorstrue

\newcommand{\versiontag}[1]{%
\par\smallskip
\noindent\fbox{\scriptsize\bfseries\sffamily #1}%
\par\nobreak\smallskip
}

\makeatletter
\newenvironment{oldversion}
{%
  \versiontag{OLD}%
  \begingroup
  \ifrevisioncolors\color{RoyalBlue}\fi
  \@afterindentfalse\@afterheading
}
{\endgroup}

\newenvironment{newversion}
{%
  \versiontag{NEW}%
  \begingroup
  \ifrevisioncolors\color{BrickRed}\fi
  \@afterindentfalse\@afterheading
}
{\endgroup}
\makeatother


\numberwithin{equation}{section}

\author[Bousquet]{Pierre Bousquet}
\address{Univ Toulouse, INSA Toulouse, CNRS, IMT, Toulouse, France.}
\email{pierre.bousquet@math.univ-toulouse.fr}

\author[Mariconda]{Carlo Mariconda}
\address{Dipartimento di Matematica "Tullio Levi-Civita", Universit\`a degli Studi di Padova, Via Trieste 63, 35121 Padova, Italy.}
\email{carlo.mariconda@unipd.it}

\author[Treu]{Giulia Treu}
\address{Dipartimento di Matematica "Tullio Levi-Civita", Universit\`a degli Studi di Padova, Via Trieste 63, 35121 Padova, Italy.}
\email{giulia.treu@unipd.it}

\title{Approximation in norm and energy with bounded functions for autonomous variational problems}

\date{\today}

\begin{document}
\begin{abstract}
For a general class of autonomous variational problems, associated to a continuous Lagrangian depending on a scalar state variable and its gradient, 
we show that every admissible function with finite energy and bounded boundary datum can be approximated, in the same Dirichlet class, by bounded admissible functions, simultaneously in \(W^{1,p}\) and in energy. 
 The key point is to consider the convex envelope of the Lagrangian with respect to the gradient variable. Exploiting the approximation techniques for  convex variational problems from~\cite{BousquetMaricondaTreu2026} and relying on a classical relaxation argument yields the desired bounded  approximation, without growth, coercivity, convexity, or structural assumptions on the Lagrangian.
\end{abstract}
\maketitle
\tableofcontents
\section{Introduction}
\subsection{Statement of the problem}
 Let \(\Omega\) be a bounded open set in \(\R^N\), let \(p\in[1,N]\), and let
\[
\varphi\in W^{1,p}(\R^N)\cap L^\infty(\R^N).
\]
We consider autonomous integral functionals of the form
\[
\mathcal L(u):=\int_\Omega H(u(x),\nabla u(x))\,dx,
\]
where
\[
H:\R\times\R^N\to\R
\]
is continuous. Throughout the paper, we use the notation
\[
\overline H(t,\xi):=|H(t,\xi)|+|\xi|^p.
\]
The admissible class is
\[
\mathcal A:=
\left\{
u\in W^{1,p}_\varphi(\Omega):
 H(u,\nabla u)\in L^1(\Omega)
\right\},
\]
where \(W^{1,p}_\varphi(\Omega)=\varphi+W^{1,p}_0(\Omega)\) and \(W^{1,p}_0(\Omega)\) is defined as the closure of smooth compactly supported functions in \(\Omega\). Then \(\mathcal{L}\) is well-defined on \(\mathcal{A}\).

\begin{definition}
We say that \(\mathcal{L}\) is \(L^{\infty}(\Omega)\) regular at some \(u\in \mathcal{A}\) when there exists a sequence \((u_k)_{k}\subset \mathcal{A}\cap L^{\infty}(\Omega)\) converging to \(u\) in \(W^{1,p}(\Omega)\) and such that \((H(u_k, \nabla u_k))_{k\geq 1}\) converges to \(H(u, \nabla u)\) in \(L^{1}(\Omega)\).
\end{definition}

We emphasize the fact that the approximating functions are required to remain in the same Dirichlet class \(W^{1,p}_\varphi(\Omega)\); in particular, the bounded approximation must preserve the prescribed boundary datum, not merely approximate \(u\) in the interior.

Deciding whether a functional \(\mathcal{L}\) is  \(L^{\infty}(\Omega)\)-regular at every \(u\in \mathcal{A}\) is often the first step to discard the Lavrentiev gap between \(W^{1,1}(\Omega)\) and \(W^{1,\infty}(\Omega)\), see e.g. \cite{BorowskiChlebicka2022, Bousquet-Annali, BousquetMaricondaTreu2014, BousquetMaricondaTreu2024,   CorboEspositoDeArcangelis, Zhikov-2006}.

In~\cite{MaricondaTreu2020}  and, recently, in~
\cite{BousquetMaricondaTreu2026}, we have formulated the following conjecture:  if \(H\) is continuous, then \(\mathcal{L}\) is \(L^{\infty}(\Omega)\)-regular at every \(u\in \mathcal{A}\). Moreover, we have established weak versions of this conjecture under a variety of additional assumptions. 
In the present paper, we show that those additional assumptions are purely technical and that mere continuity is sufficient to reach the same conclusion:

\begin{theorem}\label{teo}
Let \(H:\R\times\R^N\to\R\)
be continuous. Then \(\mathcal L\) is \(L^\infty(\Omega)\)-regular at every \(u\in\mathcal A\). 
\end{theorem}

\subsection{Ideas of the proof}
Generally speaking, approximating a function  \(u\in W^{1,p}_\varphi(\Omega)\) by bounded functions in the same Dirichlet class is often based on truncations: for every large \(k\geq 1\), one considers the bounded competitors
\[
T_k(u):=\min\{k,\max\{u,-k\}\}.
\]
On the positive tail \([u>k]\) and on the negative tail \([u<-k]\),  the gradients of \(T_k(u)\) vanish. One can prove that \(T_k(u)\) converges to \(u\) in energy; that is,
\[
\lim_{k\to +\infty} H(T_k(u),\nabla T_k(u))=H(u,\nabla u), \qquad \textrm{in } L^{1}(\Omega),
\]
when 
\(H(u,0)\) is summable, see e.g. \cite[Proposition 2.4]{BousquetMaricondaTreu2026}.
However, the latter property is not automatically implied by the finite-energy assumption \(u\in \mathcal{A}\). 
As a matter of fact,  it may even happen that \(H(u,0)\not\in L^{1}(\Omega)\) for some \(u\in \mathcal{A}\), see \cite[Example 2.10, Example 2.13]{BousquetMaricondaTreu2026}.
This is the basic obstruction behind the bounded approximation problem.

In~\cite{BousquetMaricondaTreu2026}, we  propose more general truncation constructions,  replacing the constants \(k\) by suitable functions  on \(\Omega\). Typically, we introduce a family of upper barriers \(c_{k}^+\subset \mathcal{A}\) and of lower barriers \(c_{k}^-\subset \mathcal{A}\) converging uniformly to \(+\infty\) or \(-\infty\),  and that satisfy good estimates on the tails of \(u\), namely
\[
\lim_{k\to +\infty}\int_{[u> c_{k}^{+}]}\overline{H}(c_{k}^{+}, \nabla c_{k}^{+})\,dx =0, \qquad \lim_{k\to +\infty}\int_{[ u< c_{k}^{-}]}\overline{H}(c_{k}^{-}, \nabla c_{k}^{-})\,dx =0.
\] 
Then, the corresponding truncated sequence \(u_k=\min(c_{k}^+, \max(u, c_{k}^-)\) approximates \(u\) in norm and in energy, see~\cite[Lemma 2.1]{BousquetMaricondaTreu2026}. 
Naturally, the heart of the matter is the construction of those barriers, depending on the specific properties of each Lagrangian \(H\). We proceed to present the construction of such barriers for autonomous and continuous Lagrangians, leading to the proof of Theorem~\ref{teo}.

{\bf Step 1}
When \(H\) is nonnegative and convex in the gradient variable, the functional \(\mathcal{L}\) is \(L^{\infty}(\Omega)\)-regular at every \(u\in \mathcal{A}\), see~\cite[Corollary 3.3]{BousquetMaricondaTreu2026}. 
In our present setting where \(H\) is merely continuous, 
it is thus natural to introduce the convex envelope of \(H\) with respect to the gradient variable \(\xi\in \R^N\).
Actually, it is sufficient to consider this convex envelope when evaluated at \(\xi=0\). We thus define:
\[
\Gamma(t):=\bigl(\overline H(t,\cdot)\bigr)^{**}(0),
\]
where the biconjugate is taken with respect to the gradient variable.
Exactly as in~\cite{BousquetMaricondaTreu2026}, one can prove (see~Lemma~\ref{lm-convexified-zero-gradient-cost} below) that 
\[
\Gamma\circ u \in L^{1}(\Omega).
\]

{\bf Step 2}
Assume for simplicity that \(\overline{H}\) is superlinear (which is the case when \(p>1\)). Then, by the finite-dimensional representation of the convex envelope, for every \(t\in\R\), the function \(\Gamma(t)\) can be written as 
\[
\Gamma(t)=\sum_{i=1}^{N+1}\mu_{i}^t\overline H(t,\xi_{i}^t),
\]
for some  nonnegative coefficients \(\mu_{1}^t, \dots, \mu_{N+1}^t> 0\) and vectors \(\xi_{1}^t, \dots, \xi_{N+1}^t\).
If the coefficients \(\mu_{i}^t\) were  bounded from below by a positive constant not depending on \(t\), then one could conclude at once by relying on \cite[Theorem 2.8]{BousquetMaricondaTreu2026}. Since there is no such lower bound in general, we need to introduce a new ingredient. Let us focus on the construction of upper barriers on the positive tails \([u>t]\), the lower barriers being defined similarly. Relying on a classical tool in the relaxation theory of  variational functionals \cite[Chapter X, Proposition 1.3]{EkelandTemam}, we introduce for every \(t\in \R\), a function \(b_t\in W^{1,\infty}(\Omega)\) such that
\[
\int_{[u>t]}\overline{H}(t,\nabla b_t(x))\,dx \approx \sum_{i=1}^{N+1}\int_{[u>t]}\mu_{i}^{t} \overline{H}(t,\xi_{i}^t)\,dx=|[u>t]|\Gamma(t).
\]

{\bf Step 3} 
Using that \(\liminf_{t\to +\infty}|[u>t]|\Gamma(t)=0\) (see e.g.~\cite[Lemma A.2]{BousquetMaricondaTreu2026}),  it follows from the previous step that for some sequence \((t_k)_{k\geq 1}\) converging to \(+\infty\), 
\[
\lim_{k\to +\infty} \int_{[u>t_k]}\overline{H}(t_k,\nabla b_{t_k}(x))\,dx=0.
\]
Then, the function \(c_{k}^{+}:= t_k + b_{t_k}^{+}\) is a suitable upper barrier. Lower barriers \(c_{k}^-\) are defined similarly and we can conclude that the family \(u_{k}:=\min(c_{k}^+, \max(c_{k}^-,u))\) is a bounded admissible function which approximates \(u\) in norm and in energy.

\subsection{The role of generative artificial intelligence}
After several interactions with ChatGPT-5.5 based essentially on different preliminary versions of~\cite{BousquetMaricondaTreu2026} (which also substantially modified the content and the presentation of the latter), we asked it whether the statement of Theorem~\ref{teo} (which was still a conjecture at that time) was true. It then gave a proof of the statement which, after a small number of requests to correct or clarify certain points of the argument, turned out to be valid. This proof introduced three crucial observations:
\begin{enumerate}
\item Proposition~3.1 of~\cite{BousquetMaricondaTreu2026}, formulated there for a Lagrangian \(H:\R\times\R^N\to\R^+\) which was convex at the origin, can be applied to its convex envelope \((t,\xi)\mapsto(H(t,\cdot))^{**}(\xi)\), leading to the introduction of  \(\Gamma(t)\) that satisfies the crucial property: \(\Gamma\circ u\in L^{1}(\Omega)\). 
\item One can rely on the finite-dimensional representation of the convex envelopes to translate this summability property into estimates on \(\overline{H}(t,\xi_{i}^t)\), for some vectors \(\xi_{i}^{t}\) in a finite family containing \(0\) in its convex hull (the role of those families had already been identified in \cite{BousquetMaricondaTreu2026}). 
\item Using a rather long argument inspired by some techniques in homogeneisation theory, and relying on auxiliary convergence results for oscillating functions, one can construct upper and lower barriers. 
\end{enumerate}
The proof that we present here is much shorter and technically simpler than the original proof but follows the same strategy. The main difference is related to the construction of barriers (3), that is now based on an approach reminiscent of the relaxation theory in the Calculus of Variations rather than homogeneisation (we acknowledge however the fact that both theories share many common ideas).  

The history of this proof suggests a broader reflection on the role of generative artificial intelligence in mathematical research. It may be viewed, from the authors' perspective, as an early example in which ChatGPT was not used only to rewrite arguments, improve exposition, or assemble ideas already selected by the authors, but also to provide a valid proof of a new result.

This type of interaction seems to belong to a recent stage in the use of advanced reasoning models in mathematics, and it naturally stimulates a reflection\footnote{This paragraph does not discuss the ecological, social and institutional consequences of the use of artificial intelligence, that should not be left aside  from this reflection.} on what the role of the mathematician may become in the not-so-distant future. Such tools may help to compare strategies, detect analogies between different arguments, and suggest unexpected combinations of known techniques. They may therefore accelerate the exploratory phase of research. They do not, however, replace the work of the mathematician. The formulation of the correct statement, the identification of the precise hypotheses, the detection of possible gaps, the choice of the appropriate level of generality, the explicit connections with the existing literature, the presentation of the proof emphasizing the main difficulties,  remain mathematical responsibilities. In this sense, generative artificial intelligence does not replace mathematical authorship; rather, it changes part of the exploratory process and makes even more visible the distinction between producing a proof and producing verified mathematics that can be easily understood and used by other mathematicians.

\section{Proof of Theorem~\ref{teo}}

We shall use the following standard notation from convex analysis. If \(f:\R^N\to\R\) is bounded from below by an affine function, then \(f^{**}\) denotes its biconjugate, equivalently its lower semicontinuous convex envelope. In this paper, when applied to a function of \((t,\xi)\), the operation \(^{**}\) is always taken with respect to the gradient variable \(\xi\); see, for instance,  \cite[Chapter~I, Sections~2--4]{EkelandTemam} or \cite[Part~II, Chapter~4, Section~4.8]{FonsecaLeoni}.
We mention the following representation of \(f^{**}\) when  \(f:\R^N\to \R\) is bounded from below by an affine function, see e.g.~\cite[Theorem~4.96 and Remark~4.93 (iii)]{FonsecaLeoni}:  
\begin{equation}\label{cvx-representation}
\forall \xi \in \R^N, \qquad f^{**}(\xi)=\inf \left\lbrace\sum_{i=1}^{N+1}\mu_i f(\eta_i) : \mu_i\geq 0, \eta_i\in \R^N, \sum_{i=1}^{N+1}\mu_i=1,  \sum_{i=1}^{N+1}\mu_i\eta_i=\xi\right\rbrace.
\end{equation}

The first ingredient in the proof of Theorem~\ref{teo} is given by the following result:

\begin{lemma}\label{lm-convexified-zero-gradient-cost}
Let \(H:\R\times\R^N\to\R\) be continuous, and recall that
\[
\overline H(t,\xi):=|H(t,\xi)|+|\xi|^p.
\]
Set
\[
\Gamma(t):=\bigl(\overline H(t,\cdot)\bigr)^{**}(0).
\]
If \(u\in\mathcal A\), then
\[
\Gamma\circ u\in L^1(\Omega).
\]
\end{lemma}
\begin{proof}
We apply \cite[Proposition~3.1]{BousquetMaricondaTreu2026} to the nonnegative continuous integrand \(\overline H\). Since \(u\in\mathcal A\), we have
\(\overline H(u,\nabla u)
\in L^1(\Omega)\).
For every \(t\in\R\), set
\[
G_t(\xi):=\bigl(\overline H(t,\cdot)\bigr)^{**}(\xi).
\]
Since \(\overline H\) is nonnegative, the function \(G_t\) is a finite nonnegative convex function. 
Choose
\(\zeta_t\in\partial G_t(0)\).
Then, for every \(\xi\in\R^N\),
\[
G_t(\xi)\geq G_t(0)+\langle\zeta_t,\xi\rangle
=
\Gamma(t)+\langle\zeta_t,\xi\rangle.
\]
Since \(G_t\leq\overline H(t,\cdot)\), it follows that
\[
\overline H(t,\xi)
\geq
\Gamma(t)+\langle\zeta_t,\xi\rangle
\qquad\text{for every }(t,\xi)\in\R\times\R^N.
\]
Moreover,
\[
0\leq\Gamma(t)\leq\overline H(t,0)=|H(t,0)|.
\]
Thus \(\Gamma\) is locally bounded. 
Moreover, by~\eqref{cvx-representation}, for every \(t\in\R\),
\begin{equation}\label{eq228}
\Gamma(t)=\inf \left\lbrace\sum_{i=1}^{N+1}\mu_i \overline{H}(t,\eta_i) : \mu_i\geq 0, \eta_i\in \R^N, \sum_{i=1}^{N+1}\mu_i=1,  \sum_{i=1}^{N+1}\mu_i\eta_i=0\right\rbrace
\end{equation}
so that \(\Gamma\) is Borel measurable, as the infimum of a family of continuous functions.
Therefore the assumptions of \cite[Proposition~3.1]{BousquetMaricondaTreu2026} are satisfied with
\(K:=\Gamma\) and the conclusion follows. 
\end{proof}

We can present the proof of our main result:

\begin{proof}[Proof of Theorem~\ref{teo}]
Let \(u\in\mathcal A\). Set
\[
\Gamma(t):=\bigl(\overline H(t,\cdot)\bigr)^{**}(0).
\]
By Lemma~\ref{lm-convexified-zero-gradient-cost},
\[
\Gamma\circ u\in L^1(\Omega).
\]
By \cite[Lemma A.2]{BousquetMaricondaTreu2026}, there exist two sequences \((M_{k}^+)_{k\geq 1}\) and \((M_{k}^-)_{k\geq 1}\) both converging to \(+\infty\) such that 
\[
\lim_{k\to +\infty}|[u>M_k^+]|\bigl(\Gamma(M_k^+)+1\bigr)=0, \quad
\lim_{k\to +\infty}
|[u<-M_k^-]|\bigl(\Gamma(-M_k^-)+1\bigr)=0.
\]
By regularity of the Lebesgue measure, one can find for every \(k\geq 1\) two bounded open sets 
\[
\Omega_{k}^+\supset [u>M_k^+], \quad \Omega_{k}^-\supset [u<-M_k^-]
\]
such that
\[
|\Omega_{k}^+|\leq |[u>M_k^+]|+\frac{1}{k\bigl(\Gamma(M_k^+)+1\bigr)}, \quad 
|\Omega_{k}^-|\leq |[u<-M_k^-]|+\frac{1}{k\bigl(\Gamma(-M_k^-)+1\bigr)}.
\]
Hence, 
\begin{equation}\label{eq921}
\lim_{k\to +\infty}|\Omega_{k}^+|\bigl(\Gamma(M_k^+)+1\bigr) = 0, \quad \lim_{k\to +\infty}|\Omega_{k}^-|\bigl(\Gamma(-M_k^-)+1\bigr)=0.
\end{equation}
We first construct the upper barriers. Fix \(k\geq1\). 
By the finite-dimensional representation of the convex envelope of \(\overline{H}(M_{k}^+,\cdot)\) at the origin, see~\eqref{eq228} with \(t=M_{k}^{+}\), 
there exist 
\(
\eta_{i,k}^+\in\R^N\) and 
\(
\mu_{i,k}^+>0\),
for \(i=1,\dots,N+1\),
such that
\begin{equation}\label{eq318}
\sum_{i=1}^{N+1}\mu_{i,k}^+=1,
\quad
\sum_{i=1}^{N+1}\mu_{i,k}^+\eta_{i,k}^+=0,
\end{equation}
and
\begin{equation}\label{tag:bipolar!}
\sum_{i=1}^{N+1}
\mu_{i,k}^+
\overline H(M_k^+,\eta_{i,k}^+)
\leq
\Gamma(M_k^+)+1.
\end{equation}
Let \(L_k:=\max_{1\leq i \leq N+1}|\eta_{i,k}^+|\). By uniform continuity of \(\overline{H}\) on the compact set \([M_{k}^+,M_{k}^++2]\times\overline{B}(0,L_k)\) (where \(\overline{B}(0,L_k)\) denotes the closed Euclidean ball of radius \(L_k\) and center \(0\) in \(\R^N\)), there exists \(\varepsilon_k\in (0,1/k)\) such that, for every \(\xi\in\overline{B}(0,L_k)\) and every \(t\in[M_{k}^+,M_{k}^++2\varepsilon_k]\),
\begin{equation}\label{eq973}
\left|\overline{H}(t,\xi)-\overline{H}(M_{k}^+,\xi)\right|\leq \frac{1}{k}.
\end{equation}
By \cite[Chapter X, Proposition 1.3]{EkelandTemam} applied to the function \(f(x,\xi):=\overline{H}(M_{k}^+, \xi)\) and the affine functions \(u_i(x):=\langle \eta_{i,k}^{+},x\rangle\) for \(i=1, \dots, N+1\),  on the open set \(\Omega_{k}^+\),
there exists \(b_{k}^+\in W^{1,\infty}(\Omega_{k}^+) \) that is continuous\footnote{The continuity of \(b_{k}^+\) on \(\overline{\Omega_{k}^+}\) is not explicitly stated in  \cite[Chapter X, Proposition 1.3]{EkelandTemam} but is an unambiguous outcome of its proof.} on \(\overline{\Omega_{k}^+}\) and such that for a.e. \(x\in\Omega_{k}^+\),  
\begin{equation}\label{eq978}
\left|\int_{\Omega_{k}^+}\overline{H}(M_{k}^+,\nabla b_{k}^+(x))\,dx - |\Omega_{k}^+|\sum_{i=1}^{N+1}\mu_{i,k}^+\overline{H}(M_{k}^+, \eta_{i,k}^+)\right|\leq \varepsilon_k,
\end{equation}
\begin{equation}\label{eq971}
\left|b_{k}^+(x)-\sum_{i=1}^{N+1}\mu_{i,k}^+ \langle \eta_{i,k}^+,x\rangle\right|\leq \varepsilon_k,
\end{equation} 
\begin{equation}\label{eq985}
\left|\nabla b_{k}^+(x)\right|\leq L_k,
\end{equation}
and for every \(z\in \partial \Omega_{k}^+\),
\begin{equation}\label{eq503}
b_{k}^+(z):= \sum_{i=1}^{N+1}\mu_{i,k}^+ \langle \eta_{i,k}^+,z\rangle.
\end{equation}
By \eqref{eq318}, one has \(\sum_{i=1}^{N+1}\mu_{i,k}^+ \eta_{i,k}^+=0\). Hence, \eqref{eq971} implies that
\(\|b_{k}^+\|_{L^{\infty}(\Omega_k^+)}\leq \varepsilon_k\)
while it follows from \eqref{eq503}  that \(b_{k}^{+}=0\) on \(\partial \Omega_{k}^+\). We are thus entitled to extend \(b_{k}^{+}\) by \(0\) outside \(\Omega_{k}^+\) and the resulting map, still denoted by \(b_{k}^{+}\), is Lipschitz continuous on \(\R^N\).  
Moreover, 
\[\omega_{k}^+:=\operatorname{osc} b_{k}^+=\max_{\R^N}b_{k}^+- \min_{\R^N} b_{k}^- \leq 2\varepsilon_k.
\]
By \eqref{eq973} and \eqref{eq985}, this implies that
\begin{align}
\int_{[u>M_{k}^+]}\max_{s\in [M_{k}^+, M_{k}^++\omega_{k}^+]}\overline{H}(s, \nabla b_{k}^+)\,dx
&\leq \int_{[u>M_{k}^+]} \overline{H}(M_{k}^+,\nabla b_{k}^+(x))\,dx + \frac{1}{k}|[u>M_{k}^+]|\\
&\leq \int_{\Omega_{k}^+} \overline{H}(M_{k}^+,\nabla b_{k}^+(x))\,dx + \frac{1}{k}|\Omega|,
\end{align}
where the last line follows from the fact that \(\overline{H}\) is nonnegative and \([u>M_{k}^{+}]\subset \Omega_{k}^+\cap \Omega\).
By \eqref{eq978}, this gives
\[
\int_{[u>M_{k}^+]}\max_{s\in [M_{k}^+, M_{k}^++\omega_{k}^+]}\overline{H}(s, \nabla b_{k}^+)\,dx
\leq |\Omega_{k}^+|\sum_{i=1}^{N+1}\mu_{i,k}^+\overline{H}(M_{k}^+, \eta_{i,k}^+) + \varepsilon_k + 
 \frac{1}{k}|\Omega|.
\]
In view of \eqref{tag:bipolar!}, one deduces that
\[
\int_{[u>M_{k}^+]}\max_{s\in [M_{k}^+, M_{k}^++\omega_{k}^+]}\overline{H}(s, \nabla b_{k}^+)\,dx
\leq  |\Omega_{k}^+|\left(
\Gamma(M_k^+)+1\right) + \varepsilon_k + 
 \frac{1}{k}|\Omega|.
\]
Relying on \eqref{eq921} and the fact that \(0<\varepsilon_k\leq 1/k\), one gets
\[
\lim_{k\to+\infty }\int_{[u>M_{k}^+]}\max_{s\in [M_{k}^+, M_{k}^++\omega_{k}^+]}\overline{H}(s, \nabla b_{k}^+)\,dx =0.
\]
For the negative tail, the same construction yields a sequence \((b_{k}^-)_{k\geq 1}\subset W^{1,\infty}(\Omega)\) such that
\[
\lim_{k\to+\infty }\int_{[u<-M_{k}^-]}\max_{s\in [-M_{k}^--\omega_{k}^-, -M_{k}^-]}\overline{H}(s, \nabla b_{k}^-(x))\,dx =0,
\]
where \(\omega_{k}^-:=\operatorname{osc} b_{k}^-\).
The conclusion then follows from \cite[Lemma 2.6.]{BousquetMaricondaTreu2026}.

\end{proof}
\section*{Declaration on the use of generative artificial intelligence}
During the preparation of this manuscript, the authors used ChatGPT-5.5, with high reasoning effort, as an auxiliary tool for mathematical brainstorming, proofs, and drafting support. In particular, ChatGPT has provided a first proof of Theorem~\ref{teo} applying the result of~\cite[Proposition~3.1]{BousquetMaricondaTreu2026} to the convex envelope of the Lagrangian and then using its finite-dimensional representation. The authors then worked on this proof to simplify it and to connect it with the relaxation theory, finally leading to  a much shorter argument, the one that is presented in this paper. All AI-generated suggestions were critically examined, revised, and integrated by the authors. The authors independently checked all statements, proofs, references, and conclusions, and take full responsibility for the final content of the manuscript.
\bibliographystyle{amsplain}
\bibliography{references}
\end{document}